\documentclass[reqno]{amsart}
\usepackage[colorlinks=true]{hyperref}
\hypersetup{urlcolor=blue, citecolor=red}

\usepackage{euscript}

\newtheorem{Theorem}{Theorem}[section]

\newtheorem{Lemma}[Theorem]{Lemma}
\newtheorem{Proposition}[Theorem]{Proposition}
\theoremstyle{definition}

\def \R{\mathbb R}
\def \p{\partial}
\def \MP{\mathcal{MP}}
\def\g{\gamma}
\def \g{\gamma}
\def \s{\sigma}
\def \pxi{\frac{\partial}{\partial x^i}}
\def \<{\langle}
\def \>{\rangle}

\renewcommand{\(}{\left(}
\renewcommand{\)}{\right)}

\newcommand{\grad}{\operatorname{grad}}
\newcommand{\hess}{\operatorname{Hess}}
\newcommand{\de}[2]{\frac{\partial #1}{\partial #2}}

\begin{document}

\title[Boundary and scattering rigidity problems for $\MP$-systems]{Boundary and scattering rigidity problems in the presence of a magnetic field and a potential}

\author[Yernat M. Assylbekov]{Yernat M. Assylbekov}
\address{Department of Mathematics,
University of Washington,
Seattle, WA 98195-4350, USA}
\email{y\_assylbekov@yahoo.com}

\author[Hanming Zhou]{Hanming Zhou}
\address{Department of Mathematics,
University of Washington,
Seattle, WA 98195-4350, USA}
\email{hzhou@math.washington.edu}

\begin{abstract}
In this paper, we consider a compact Riemannian manifold with boundary, endowed with a magnetic potential $\alpha$ and a potential $U$. For brevity, this type of systems are called $\MP$-systems. On simple $\MP$-systems, we consider both the boundary rigidity problem and scattering rigidity problem, see the introduction for details. We show that these two problems are equivalent on simple $\MP$-systems. Unlike the cases of geodesic or magnetic systems, knowing boundary action functions or scattering relations for only one energy level is insufficient to uniquely determine a simple $\MP$-system, even under the assumption that we know the restriction of the system on the boundary $\p M$, and we provide some counterexamples. These problems can only be solved up to an isometry and a gauge transformations of $\alpha$ and $U$. We prove rigidity results for metrics in a given conformal class, for simple real analytic $\MP$-systems and for simple two-dimensional $\MP$-systems.
\end{abstract}
\maketitle

\section{Introduction}
\subsection{Posing the problems}
Given a Riemannian manifold $(M,g)$ of dimension $n\ge2$ with boundary, endowed with a magnetic field $\Omega$, that is a closed 2-form, we consider the law of motion described by the Newton's equation
\begin{equation}\label{MP}
\nabla_{\dot\gamma}\dot\gamma=Y(\dot\gamma)-\nabla U(\gamma),
\end{equation}
where $U$ is a smooth function on $M$, $\nabla$ is the Levy-Civita connection of $g$ and $Y:TM\to TM$ is the \emph{Lorentz force} associated with $\Omega$, i.e., the bundle map uniquely determined by
\begin{equation}\label{lorentz}
\Omega_x(\xi,\eta)=\langle Y_x(\xi),\eta\rangle
\end{equation}
for all $x\in M$ and $\xi,\eta\in T_xM$. A curve $\gamma:[a,b]\to M$, satisfying \eqref{MP} is called an \emph{$\MP$-geodesic}. The equation \eqref{MP} defines a flow $\phi_t$ on $TM$ that we call an \emph{$\MP$-flow}. These are not standard terms in general. Note that time is not reversible on the $\MP$-geodesics, unless $\Omega=0$.

When $\Omega=0$ the flow is called {\it potential flow}; while if $U=0$ we obtain the {\it magnetic flow}. Therefore, the equation \eqref{MP} describes the motion of a particle on a Riemannian manifold under the influence of a magnetic field $\Omega$ in a potential field $U$. Magnetic flows were firstly considered in \cite{AS, Ar} and it was shown in \cite{ArG, Ko, N1, N2, NS, PP} that they are related to dynamical systems, symplectic geometry, classical mechanics and mathematical mechanics.

When $\Omega$ is exact, i.e. $\Omega=d\alpha$ for some magnetic potential $\alpha$, the $\MP$-flow also arises as the Hamiltonian flow of $H(x,p)=\frac12(p+\alpha)_{g(x)}^2+U(x)$ with respect to the canonical symplectic form of $T^*M$.

For $\MP$-flow the energy $E(x,\xi)=\frac{1}{2}|\xi|^2+U(x)$ is an integral of motion. By the Law of Conservation of Energy, for every $\MP$-geodesic the energy is constant along it. Unlike the geodesic flow, where the flow is the same (up to time scale) on any
energy levels, $\MP$-flow depends essentially on the choice of the energy
level. Throughout the paper we assume the energy level $k>\sup_{x\in M}U(x)$ with $S^kM=E^{-1}(k)$, \emph{the bundle of energy $k$}.
Note that it is necessary for $k$  to be strictly greater than the supremum of $U$. Because otherwise we would get that at some $x\in M$ any vector $\xi\in S^k_x M$ has non-positive length.

We define the action $\mathbb A(x,y)$ between boundary points as
a minimizer of the appropriate action functional, see (\ref{Axy})
and Appendix~\ref{Mane-action}. In the case $\Omega=0$ and $U=0$, the function
$\mathbb A(x,y)$ coincides with the boundary distance function $d_g(x,y)$.
In this case, we cannot recover $g$ from $d_g$ up to isometry,
unless some additional assumptions are imposed on $g$, see, e.g., \cite{Cr1}. One such assumption is the simplicity of the metric, see, e.g., \cite{Mi, Sh1, SU, SU1}. We consider below the analog of simplicity for $\MP$-systems.

Let $\Lambda$ denotes the second fundamental form of $\p M$, and $\nu(x)$ denotes the inward unit vector normal to $\p M$ at $x$. We say that $\p M$ is \emph{strictly $\MP$-convex} if
\begin{equation}\label{strict-conv}
\Lambda(x,\xi)>\langle Y_x(\xi),\nu(x)\rangle-d_xU(\nu(x))
\end{equation}
for all $(x,\xi)\in S^k(\p M)$.

For $x\in M$, we define the {\em $\MP$-exponential map} at $x$ to be the partial map
$\exp^{\MP}_x:T_xM\to M$ given by
$$
\exp^{\MP}_x(t\xi)=\pi\circ\phi_t(\xi),\quad t\ge 0,\ \xi\in S^k_xM.
$$
It is not hard to show that, for every $x\in M$,
$\exp^{\MP}_x$  is a $C^1$-smooth partial map on $T_xM$
which is $C^\infty$-smooth on $T_xM\setminus\{0\}$.

We say that $M$ is {\em simple} w.r.t. $(g,\Omega,U)$ if $\partial M$ is strictly $\MP$-convex and the $\MP$-exponential map
$\exp^{\MP}_x:(\exp^{\MP}_x)^{-1}(M)\to M$
is a diffeomorphism for every $x\in M$. In this case, $M$ is diffeomorphic to the unit ball of $\R^n$. Therefore $\Omega$ is exact, and we let $\alpha$ be a magnetic potential, i.e. $\alpha$ is a 1-form on $M$ such that
$$
d\alpha=\Omega.
$$
Henceforth we call $(g,\alpha,U)$ a {\it simple $\MP$-system} on $M$. We will also say that $(M,g,\alpha,U)$ is a simple $\MP$-system. It is easy to see that the simplicity is stable under a small perturbation of the energy level.

First, we state the boundary rigidity problem. Given $x,y\in M$, let
$$
\mathcal C(x,y)=\{\gamma:[0,T]\to M:T>0,\gamma(0)=x,\gamma(T)=y, \text{ $\gamma$ is absolutely continuous}\}.
$$
The {\it time free action} of a curve $\gamma\in \mathcal C(x,y)$ w.r.t. $(g,\alpha,U)$ is defined as
$$
\mathbb A(\gamma)=\frac{1}{2}\int_0^T|\dot\gamma(t)|^2\,dt+kT-\int_{\gamma}(\alpha+U)
$$

For a simple $\MP$-system, $\MP$-geodesics with energy $k$ minimize the time free action (see Appendix \ref{Mane-action})
\begin{equation}\label{Axy}
\mathbb A(x,y):=\inf_{\gamma\in\mathcal C(x,y)}\mathbb A(\gamma)=2kT_{x,y}-\int_{\gamma_{x,y}}(\alpha+2U),
\end{equation}
where $\gamma_{x,y}:[0,T_{x,y}]$ is the unique $\MP$-geodesic with constant energy $k$ from $x$ to $y$. The function $\mathbb A(x,y)$ is referred to as Ma\~n\'e's action potential (of energy $k$), and we call the restriction $\mathbb A|_{\p M\times\p M}$ the {\it boundary action function}.

We say that two $\MP$-systems $(g,\alpha,U)$ and $(g',\alpha',U')$ are {\it gauge equivalent} if there is a diffeomorphism $f:M\to M$, which is the identity on the boundary, and a smooth function $\varphi:M\to \R$, vanishing on the boundary, such that $g'=f^*g$, $\alpha'=f^*\alpha+d\varphi$ and $U'=U\circ f$. Observe that given two gauge equivalent $\MP$-systems, if one is simple, then the other one is also simple. Moreover, if two simple $\MP$-systems are gauge equivalent, then they have the same boundary action function.

The {\it boundary rigidity problem in the presence of a magnetic field and a potential} studies that to which extend an $\MP$-system $(g,\alpha,U)$ on $M$ is determined by the boundary action functions. By the above observation, one can only expect to obtain the uniqueness up to gauge equivalence. For the zero potential, i.e. $U=0$, we obtain the boundary rigidity problem for the magnetic systems that was considered by N. Dairbekov, G. Paternain, P. Stefanov and G. Uhlmann in \cite{DPSU}. In the absence of both magnetic fields and potentials, i.e. $\Omega=0$ and $U=0$, we come to the ordinary boundary rigidity problem for the Riemannian metrics. For the recent surveys on the ordinary boundary rigidity problem see \cite{Cr2,SU2}. It also worths to mention that recently P. Stefanov, G. Uhlmann and A. Vasy \cite{SUV} proved the boundary rigidity with partial data for metrics in a given conformal class, this is so far the only local boundary rigidity result.

Next, we define a scattering relation and state the {\it scattering rigidity problem in the presence of a magnetic field and a potential}. Let $\p_+ S^k M$ and $\p_- S^k M$ denote the bundles of inward and outward vectors of energy $k$ over $\p M$
\begin{align*}
\p_\pm S^k M=\{(x,\xi)\in S^k M:x\in \p M, \pm\langle\xi,\nu(x)\rangle\ge 0\}
\end{align*}
where $\nu$ is the inward unit vector normal to $\p M$. For $(x,\xi)\in \p_+ S^k M$ let $\tau(x,\xi)$ be the time when the $\MP$-geodesic $\gamma_{x,\xi}$, such that $\gamma_{x,\xi}(0)=x$,
$\dot\gamma_{x,\xi}(0)=\xi$, exits. By
Lemma \ref{l-l} the function $\tau:\p_+ S^kM\to \R$ is smooth.

The scattering relation $\mathcal S:\p_+ S^k M\to\p_-S^k M$ of an $\MP$-system $(M,g,\alpha,U)$ is defined as
$$
\mathcal S(x,\xi)=(\phi_{\tau_+(x,\xi)}(x,\xi))=(\gamma_{x,\xi}(\tau_+(x,\xi)),\dot\gamma_{x,\xi}(\tau_+(x,\xi))).
$$

Observe that two gauge equivalent $\MP$-systems have the same scattering relation. Is this the only type of nonuniqueness? In other words, the scattering rigidity problem studies whether a simple $\MP$-system $(M,g,\alpha,U)$, up to gauge equivalence, is uniquely determined by the scattering relations. In the Euclidean space this problem was considered by R. G. Novikov \cite{N}, in the absence of magnetic field, and by A. Jollivet \cite{Jol}. On Riemannian manifolds endowed with magnetic fields, scattering rigidity problem was studied by N. Dairbekov, G. Paternain, P. Stefanov and G. Uhlmann in \cite{DPSU}, by P. Herreros in \cite{H}, and by P.~Herreros and J. Vargo in \cite{HJ}. The reconstruction of both the Riemannian metrics and magnetic fields from the scattering relations was considered by N. Dairbekov and G. Uhlmann \cite{DU} for simple two-dimensional magnetic systems.

For simple $\MP$-systems, the boundary rigidity and the scattering rigidity problems are equivalent, see Lemma \ref{jet} and Theorem~\ref{equivalence of two problems}. Therefore, we formulate all rigidity results in terms of the boundary rigidity problem. However, not like the boundary rigidity problems for simple manifolds or simple magnetic systems ( with energy $1/2$), the boundary rigidity problem for simple $\MP$-systems needs the information of the boundary action functions for two different energy levels, see the counterexamples in Section 3 and the proofs of the main results in Section 5 for details.

We consider these problems under various natural restrictions: simple $\MP$-systems with metrics in a given conformal class, simple real-analytic $\MP$-systems and simple two-dimensional $\MP$-systems.

\subsection{Structure of the paper}
The rest of the paper is organized as follows. In Section 2, we show by doing the change of metrics, one can reduce a simple $\MP$-system to a simple magnetic system with the same boundary action function. Section 3 provides counterexamples which show that knowing the boundary action function for only one energy level is insufficient for solving the boundary rigidity problem, even under the assumption that the restriction of the system on the boundary $\p M$ is known. In Section 4, we demonstrate the equivalence between the boundary rigidity problem and the scattering rigidity problem for a simple $\MP$-system. Section 5 is devoted to the proofs of the boundary rigidity for various systems, namely, simple $\MP$-systems with metrics in a given conformal class, simple real-analytic $\MP$-systems and simple two-dimensional $\MP$-systems. We give a final remark on the case that we only know the boundary action function for one energy level in Section 6.

\subsection{Acknowledgements}

The authors thank their advisor Prof. Gunther Uhlmann for helpful suggestions and reading an earlier version of this paper. The authors also thank Prof. Nurlan Dairbekov for the suggestion on the possible approach for this problem. The work of both authors was partially supported by NSF.

\section{Relation between $\MP$-systems and magnetic systems}
\subsection{Reduction to the magnetic system}
For a fixed energy level $k>\sup_{x\in M}U(x)$, let $\sigma(t)$ be an $\MP$-geodesic with the constant energy $k$. Consider the time change
$$
s(t)=\int_0^t 2(k-U(\sigma))\,dt.
$$
Then $s$ is the arclength of $\g(s)=\s(t(s))$ under the metric $G=2(k-U)g$. The following version of Maupertuis' principle says that $\gamma(s)=\sigma(t(s))$ is a unit speed magnetic geodesic of the magnetic system $(G,\alpha)$.

\begin{Theorem}\label{MP system is magnetic system}
Let $(g,\alpha,U)$ be an $\MP$-system on $M$ and let $k$ be a constant such that $k>\sup_{M} U$. Suppose $\sigma(t)$ is an $\MP$-geodesic of energy $k$. Then $\gamma(s)=\sigma(t(s))$ is a unit speed magnetic geodesic of the magnetic system $(G,\alpha)$.
\end{Theorem}
\begin{proof}
It is immediate to check that $\gamma$ has unit speed with respect to $G$. Let $\rho$ denote the arclength of the metric $g$. Since we fix the energy to be $k$, the parameter $t$ of $\sigma$ must be proportional to the length, i.e. $dt=d\rho/\sqrt{2(k-U)}$. We denote by $\dot{\gamma}$ the derivative of $\gamma$ with respect to $s$ and by $\dot\sigma$ the derivative of $\sigma$ with respect to $t$. By the Maupertuis' principle, the $\MP$-geodesic is an extremal of the action
\begin{align*}
\int_\sigma \frac{\p L}{\p \xi^i}(\sigma,\dot{\sigma})\dot\sigma^i\,dt&
=\int_\sigma \sqrt{2(k-U(\sigma))}\,d\rho-\int_\sigma\alpha\left(\sigma,\frac{d\sigma}{d\rho}\right)\,d\rho\\
&=\int_{\gamma} \,ds-\int_{\gamma}\alpha(\gamma,\dot\gamma)\,ds.
\end{align*}
Hence the Lagrangian $L(x,\xi)=|\xi|_{G(x)}-\alpha_x(\xi)$ satisfies the Euler-Lagrange equation with respect to $s$ which has the form
$$
\frac{d}{ds}\left(\frac{G_{ki}\dot\gamma^i}{|\dot\gamma|_{G}}-\alpha_k\right)=
\frac{1}{2}\frac{1}{|\dot\gamma|_{G}}\de{G_{ij}}{x^k}\dot\gamma^i \dot\gamma^j-\de{\alpha_i}{x^k}\dot\gamma^i.
$$
Since $s$ is the arclength of $G$, for which $|\dot\gamma|_{G}=1$, this equation takes the form
$$
\frac{d}{ds}\left(G_{ki}\dot\gamma^i-\alpha_k\right)=
\frac{1}{2}\de{G_{ij}}{x^k}\dot\gamma^i \dot\gamma^j-\de{\alpha_i}{x^k}\dot\gamma^i.
$$
Taking the derivative with respect to $s$ and multiplying by $G^{mk}$ we have
$$
\ddot\gamma^m+G^{mk}\left(\de{G_{ki}}{x^j}+\frac12\de{G_{ij}}{x^k}\right)\dot\gamma^i\dot\gamma^j=G^{mk}\left(\de{\alpha_k}{x^i}-\de{\alpha_i}{x^k}\right)\dot\gamma^i,
$$
which is the equation of magnetic geodesics of the magnetic system $(G,\alpha)$.
\end{proof}

We give an alternative proof of Theorem \ref{MP system is magnetic system} based on the flow equation itself.

\begin{proof}
Given an $\MP$-geodesic $\s(t)$ with energy $k$ and a positive smooth function $\phi$, let $G=\phi g$, $ds=\sqrt{2\phi(k-U)}dt$, so $s$ will be the arclength of $\g(s)=\s(t(s))$ under the metric $G$. If we denote the Christoffel symbols and the covariant derivative under the new metric $G$ by $\widetilde{\Gamma}^i_{jk}$ and $\widetilde{D}$ respectively, then
$$\widetilde{\Gamma}^i_{jk}=\Gamma^i_{jk}+\frac{1}{2}\phi^{-1}(\delta^i_k\frac{\p \phi}{\p x^j}+\delta^i_j\frac{\p \phi}{\p x^k}-\phi^{,i}g_{jk}).$$
So 
$$\frac{\widetilde{D}\dot\gamma}{ds}=\ddot\gamma^i\frac{\p}{\p x^i}+\dot\g^j\dot\g^k\widetilde{\Gamma}^i_{jk}\frac{\p}{\p x^i}$$
$$
=\ddot\s^i\(\frac{dt}{ds}\)^2\frac{\p}{\p x^i}+\dot\s^i\frac{d^2t}{ds^2}\pxi+(\frac{dt}{ds})^2\dot\s^j\dot\s^k(\Gamma^i_{jk}+\frac{1}{2}\phi^{-1}(\delta^i_k\frac{\p \phi}{\p x^j}+\delta^i_j\frac{\p \phi}{\p x^k}-\phi^{,i}g_{jk}))\pxi$$
$$=(\frac{dt}{ds})^2\frac{D\dot\s}{dt}+(\frac{dt}{ds})^2(\phi^{-1}\<\nabla\phi,\dot\s\>_g\dot\s-\frac{1}{2}\phi^{-1}|\dot\s|_g^2\nabla\phi)+\frac{d^2t}{ds^2}\dot\s$$
Here $\frac{d^2t}{ds^2}=-\frac{1}{2}\frac{1}{(2\phi(k-U))^2}\frac{d(2\phi(k-U))}{dt}=-\frac{1}{2}(\frac{dt}{ds})^2\frac{1}{2\phi(k-U)}\frac{d(2\phi(k-U))}{dt}$\\
$$=-(\frac{dt}{ds})^2\{\frac{1}{2}\phi^{-1}\<\nabla\phi,\dot\s\>_g+\frac{1}{2(k-U)}\<\nabla(k-U),\dot\s\>_g\},$$ 
thus
$$\frac{\widetilde{D}\dot\g}{ds}=(\frac{dt}{ds})^2\{Y(\dot\s)-\nabla U(\s)+\phi^{-1}\<\nabla\phi,\dot\s\>_g\dot\s-\phi^{-1}(k-U)\nabla\phi$$
$$-(\frac{1}{2}\phi^{-1}\<\nabla\phi,\dot\s\>_g+\frac{1}{2(k-U)}\<\nabla(k-U),\dot\s\>_g)\dot\s\}.$$
Now we let $\phi=2(k-U)$, so $ds=2(k-U)dt$ (Actually all $\phi=c(k-U), c\in \R^+$ will work for our argument), we get 
$$\frac{\widetilde{D}\dot\g}{ds}=(\frac{dt}{ds})^2 Y(\dot\s)=\frac{dt}{ds} Y(\dot\g).$$
This indeed gives us the magnetic flow with the Lorentz force $Y_G=\frac{1}{2(k-U)} Y$. Moreover, one can see that the magnetic potential $\widetilde{\alpha}$ associated to $(G, Y_G)$ is $\alpha$ too, i.e. the new magnetic system is $(G, \alpha)$.
\end{proof}

\subsection{Simplicities of two systems}The next result says that the simplicity of $(g,\alpha,U)$ implies the simplicity of $(G,\alpha)$, and vice versa. A simple magnetic system is a special case of simple $\MP$-systems by assuming the potential $U=0$ and the energy $k=\frac{1}{2}$, see\cite{DPSU} for more details. 
\begin{Proposition}\label{equivalence of simplicities}
The $\MP$-system $(g,\alpha,U)$ on $M$ (of energy $k$) is simple if and only if so is the magnetic system $(G,\alpha)$ (of energy $\frac{1}{2}$).
\end{Proposition}
\begin{proof}
Since the trajectories of these two systems coincide, for every $x\in M$ the $\MP$-exponential map $\exp_x^\MP:(\exp_x^\MP)^{-1}(M)\to M$ is a diffeomorphism if and only if the magnetic exponential map $\exp_x^\mu:(\exp_x^\mu)^{-1}(M)\to M$ is a diffeomorphism (The definition of $\exp_x^\mu$ is similar to $\exp_x^\MP$ by replacing the $\MP$-flow with a magnetic flow of energy $\frac{1}{2}$). Hence, it is sufficient to prove that $\p M$ is strictly $\MP$-convex if and only if it is strictly magnetic convex with respect to $(G,\alpha)$.

First, we introduce some notations. The inward unit vector normal to $\p M$ with respect to the metric $G$ is indicated as $n$, thus $n=(2(k-U))^{-\frac{1}{2}}\nu$ ($G$ is conformal to $g$). The unit sphere bundle of the metric $G$ is denoted by $SM$. By $\Lambda_G$ we denote the second fundamental form of $\p M$ with respect to metric $G$. From the definition of the second fundamental form and using the formula for connection of $G$ in terms of connection of $g$, we obtain the following formula:
\begin{equation}\label{main equation in equivalence of simplicities proposition 1}
\Lambda_G(x,\xi)=\sqrt{2(k-U(x))}\Lambda(x,\xi)+\frac{d_xU(\nu)}{\sqrt{2(k-U(x))}}|\xi|^2,\quad x\in\p M, \xi\in T_x\p M.
\end{equation}
The Lorentz force of the magnetic field $d\alpha$ with respect to the metric $G$ is indicated as $Y_G$. The next formula is obvious,
\begin{equation}\label{main equation in equivalence of simplicities proposition 2}
\langle Y_G(\xi),n\rangle_{G(x)}=\frac{1}{\sqrt{2(k-U(x))}}\langle Y(\xi),\nu \rangle,\quad x\in\p M, \xi\in T_x\p M.
\end{equation}

Now, suppose that $\p M$ is strictly $\MP$-convex. Take any $x\in \p M$ and $v\in S_x(\p M)$ for the metric $G$. We substitute the vector $\xi=2(k-U)v\in S^k_x(\p M)$ for the metric $g$ in formulas (\ref{main equation in equivalence of simplicities proposition 1}--\ref{main equation in equivalence of simplicities proposition 2}) to obtain
\begin{align*}
&\Lambda_G(x,v)=(2(k-U(x)))^{-\frac{3}{2}}(\Lambda(x,\xi)+d_xU(\nu)),\\
&\langle Y_G(v),n\rangle_{G(x)}=(2(k-U(x)))^{-\frac{3}{2}}\langle Y(\xi),\nu \rangle,
\end{align*}
which implies that $\p M$ is strictly magnetic convex with respect to $(G,\alpha)$ by \eqref{strict-conv}.

By similar arguments one can show that $\p M$ is strictly $\MP$-convex whenever it is strictly magnetic convex with respect to $(G,\alpha)$.
\end{proof}

\subsection{Boundary action functions of the two systems}
Here we show that the boundary action functions of the two simple systems $(g,\alpha,U)$ and $(G,\alpha)$ coincide. Assuming the potential $U=0$ and the energy $k=\frac{1}{2}$, the corresponding boundary action function of \eqref{Axy} is the one for a simple magnetic system.
\begin{Proposition}\label{A=A_G}
Let $\mathbb A$ be the Ma\~n\'e's action potential (of energy $k$) for a simple $\MP$-system $(g,\alpha, U)$ and $\mathbb A_G$ be the Ma\~n\'e's action potential (of energy $1/2$) for the simple magnetic system $(G,\alpha)$, then $\mathbb A|_{\p M\times\p M}=\mathbb A_G|_{\p M\times\p M}$.
\end{Proposition}
\begin{proof}
Take $x,y\in\p M$ and consider the unique $\MP$-geodesic $\sigma$ from $x$ to $y$. Then Theorem \ref{MP system is magnetic system} implies that $\gamma(s)=\sigma(t(s))$ is a unit speed magnetic geodesic (from $x$ to $y$) of the system $(G,\alpha)$ and $s$ is the arclength of $\gamma$ under the metric $G$. Thus,
\begin{multline*}
\mathbb A(x,y)=2kT(x,y)-\int_{\sigma}(\alpha(\sigma,\dot{\sigma})+2U(\sigma))\,dt\\
=\int_{\sigma}2(k-U(\sigma))\,dt-\int_{\sigma}\alpha(\sigma,\dot{\sigma})\,dt
=\int_{\gamma}\,ds-\int_{\gamma}\alpha(\gamma,\dot{\gamma})\,ds\\
=T_G(x,y)-\int_{\gamma}\alpha(\gamma,\dot{\gamma})\,ds=\mathbb A_G(x,y).
\end{multline*}
We are done.
\end{proof}


\section{Counterexamples}

Before moving to the detailed study of the boundary and scattering rigidity problems of simple $\MP$-systems, we provide some counterexamples which show that knowing the boundary action function for only one energy level is insufficient for solving the boundary rigidity problem, even under the assumption that we know the restriction of the system on the boundary $\p M$. More precisely, there are simple $\MP$-systems $(g,\alpha, U)$ and $(g',\alpha',U')$ with the same boundary action function for some energy level $k$, whose restrictions onto the boundary are the same (i.e. $g|_{\p M}=g'|_{\p M}, \alpha|_{\p M}=\alpha'|_{\p M}, U|_{\p M}=U'|_{\p M}$), but are not gauge equivalent. This makes one turn to considering boundary action functions of two different energy levels.\\

\noindent {\bf Counterexamples:} Given some simple magnetic system $(g, \alpha)$ on a compact manifold $M$ with boundary, we define two $\MP$-systems $(\frac{1}{4}g,\alpha, U_1)$ and $(\frac{1}{2}g,\alpha, U_2)$, where $U_1\equiv 1$ on $M$ and $U_2\equiv 2$ on $M$. We fix the energy $k=3$, then it is easy to see that these two $\MP$-systems reduce to the same magnetic system $(g, \alpha)$. Since $(g,\alpha)$ is simple, Proposition \ref{equivalence of simplicities} implies that both $(\frac{1}{4}g,\alpha, U_1)$ and $(\frac{1}{2}g,\alpha, U_2)$ are simple $\MP$-systems. Moreover, appling Proposition \ref{A=A_G}, we conclude that they have the same boundary action function for the energy $k=3$. Obviously these two $\MP$-systems are not gauge equivalent, since the metrics $\frac{1}{4}g$ and $\frac{1}{2}g$, potentials $U_1$ and $U_2$ are even not equal on the boundary $\p M$.

Next, by modifying the two $\MP$-systems near the boundary, we can make them equal on the boundary. Let $\varphi$ and $\psi$ be two smooth functions on $M$, and $\varphi\equiv\psi\equiv 1$ for points away from a small tubular neighborhood of the boundary $\p M$. We assume $1\leq \varphi<\frac{3}{2}$ in the interior of $M$ and $\varphi=\frac{3}{2}$ on $\p M$; $\frac{3}{4}<\psi\leq 1$ in the interior of $M$ and $\psi=\frac{3}{4}$ on $\p M$. Then $\varphi=\varphi U_1<\frac{3}{2}<\psi U_2=2\psi$ in the interior of $M$. We define $\tilde g=\frac{1}{2(3-\varphi)}g$ and $\tilde g'=\frac{1}{2(3-2\psi)}g$. Then it is easy to check that the $\MP$-systems $(\tilde g,\alpha, \varphi)$ and $(\tilde g',\alpha, 2\psi)$ reduce to the same magnetic system $(g, \alpha)$ for the energy $k=3$. Applying Proposition \ref{equivalence of simplicities} and \ref{A=A_G} again, these two $\MP$-systems $(\tilde g,\alpha, \varphi)$ and $(\tilde g',\alpha, 2\psi)$ are simple with the same boundary action function. Moreover, $\tilde g|_{\p M}=\tilde g'|_{\p M}=\frac{1}{3}g, \varphi|_{\p M}=2\psi|_{\p M}=\frac{3}{2}$, i.e. these two systems are equal on the boundary. However, they are still not gauge equivalent, there is no diffeomorphism $f: M\rightarrow M$ such that $2\psi=\varphi\circ f$ (since $\varphi<\frac{3}{2}<2\psi$ in the interior of $M$).


\section{Boundary action function and scattering relation}
\subsection{Boundary determination}
Here we show that up to gauge equivalence the boundary action functions of two different energy levels completely determine the Riemannian metric, magnetic potential and potetial on the boundary of the manifold under study. As mentioned in the Section 3, the boundary action function of one energy level is insufficient for determining the restriction of the system on the boundary.

\begin{Lemma}\label{bnd-0}
If $(g,\alpha,U)$ and $(g',\alpha',U')$ are simple $\MP$-systems
on $M$ with the same boundary action functions for both energy $k_1$ and $k_2$ ,
then
\begin{equation}\label{bnd-data-0}
\imath^*g=\imath^*g',
\quad
\imath^*\alpha=\imath^*\alpha',
\quad
U\circ \imath=U'\circ \imath,
\end{equation}
where $\imath:\partial M\to M$ is the embedding map.
\end{Lemma}

\begin{proof}
Given $x\in\partial M$ and $\xi\in T_x(\partial M)$,
let $\tau (s)$, $-\varepsilon <s<\varepsilon $, be a curve on
$\partial M$ with $\tau (0)=x$ and $\dot \tau(0)=\xi$. 
Let $G=2(k_1-U)g$, by Theorem \ref{MP system is magnetic system} and Proposition \ref{equivalence of simplicities}, $(G,\alpha)$ is a simple magnetic system of energy $\frac{1}{2}$. Applying Proposition \ref{A=A_G} it is easy to see that
$$
\lim_{s\to0}\frac{\mathbb A(x,\tau(s))}s=\lim_{s\to 0}\frac{\mathbb A_{G}(x,\tau(s))}{s}=|\xi|_{G}-\alpha(\xi)
=\sqrt{2(k_1-U)}|\xi|_g-\alpha(\xi).
$$

A similar equality holds for the system  $(g',\alpha', U')$.
Therefore,
$$
\sqrt{2(k_1-U)}|\xi|_{g}-\alpha(\xi)=\sqrt{2(k_1-U')}|\xi|_{g'}-\alpha'(\xi).
$$
Changing $\xi$ to $-\xi$, we get
$$
\sqrt{2(k_1-U)}|\xi|_{g}+\alpha(\xi)=\sqrt{2(k_1-U')}|\xi|_{g'}+\alpha'(\xi),
$$
whence we infer the second equation in \eqref{bnd-data-0}. Notice that we also get that 
$$(k_1-U)|\xi|_g^2=(k_1-U')|\xi|_{g'}^2.$$
Similarly, for energy $k_2$, we obtain
$$(k_2-U)|\xi|_g^2=(k_2-U')|\xi|_{g'}^2.$$
Since $k_1\neq k_2$, by taking the difference of above two equations, we have $|\xi|_g=|\xi|_{g'}$, thus $\imath^*g=\imath^*g'$. This also implies that $U\circ \imath=U'\circ \imath$.
\end{proof}

Now we prove that the boundary action functions of two different energy levels actually determine the full jets of the metric $g$, magnetic potential $\alpha$ and potential function $U$ on the boundary.
 
\begin{Lemma}\label{jet}
If $(g,\alpha,U)$ and $(g',\alpha',U')$ are simple $\MP$-systems on $M$ with the same boundary action functions for both energy $k_1$ and $k_2$, then $(g',\alpha',U')$ is gauge equivalent to some simple $\MP$-system $(\bar{g},\bar{\alpha},\bar{U})$ such that in any local coordinate system we have $\p^m g|_{\p M}=\p^m \bar{g}|_{\p M}, \p^m \alpha|_{\p M}=\p^m \bar{\alpha}|_{\p M}$ and $\p^m U|_{\p M}=\p^m \bar{U}|_{\p M}$ for every multi-index $m$.
\end{Lemma}
\begin{proof}
Let $G_i=2(k_i-U)g$ and $G'_i=2(k_i-U')g'$, $i=1,2$ by Theorem \ref{MP system is magnetic system} and Proposition \ref{equivalence of simplicities}, $(G_i,\alpha)$ and $(G'_i,\alpha')$ are simple magnetic systems of energy $\frac{1}{2}$. Let $\mathbb A_{G_i}$ and $\mathbb A_{G'_i}$ denote the Ma\~n\'e's action potentials (of energy $1/2$) for $(G_i,\alpha)$ and $(G'_i,\alpha')$ respectively. Then by Proposition \ref{A=A_G} we have $\mathbb A_{G_i}|_{\p M\times \p M}=\mathbb A_{G'_i}|_{\p M\times \p M}$. Then \cite[Theorem 2.2]{DPSU} implies that there is $(\bar{G_i},\bar{\alpha_i})$, gauge equivalent to $(G'_i,\alpha')$, such that in any local coordinate system $\p^m G_i|_{\p M}=\p^m \bar{G_i}|_{\p M}$ and $\p^m\alpha|_{\p M}=\p^m\bar{\alpha_i}|_{\p M}$ for every multi-index $m$. Thus there is some diffeomorphism $f_i$ with $f_i|_{\p M}=Id$, and some smooth function $\varphi_i$ with $\varphi_i|_{\p M}=0$, such that $\bar{G_i}=f_i^*G'_i=2(k_i-U'\circ f_i)f_i^*g'$ and $\bar{\alpha_i}=f_i^*\alpha'+d\varphi_i$. Actually by Lemma \ref{bnd-0} $U|_{\p M}=U'|_{\p M}$, the proof of \cite[Theorem 2.2]{DPSU} showes that near the boundary $\p M$, one can choose $f_1=f_2=\exp_{\p M}\circ (\exp'_{\p M})^{-1}$ where $\exp_{\p M}$ and $\exp'_{\p M}$ are the ``usual" boundary exponential maps w.r.t. $g$ and $g'$ respectively. Thus $f_1^*g'=f_2^*g',\, U'\circ f_1=U'\circ f_2$ near the boundary. We define $\bar{g}=f_1^*g', \, \bar{\alpha}=\bar{\alpha_1},\, \bar{U}=U'\circ f_1$, by Lemma \ref{bnd-0}, $U|_{\p M}=U'_{\p M}=\bar{U}|_{\p M}$. Thus $G_1|_{\p M}=\bar{G_1}|_{\p M}$ implies $g|_{\p M}=\bar{g}|_{\p M}$. 

Now we prove the equality of derivatives on the boundary by introducing boundary normal coordinates $(x', x^n)$ w.r.t. $g$ near arbitrary $x_0\in\p M$. Since $g|_{\p M}=\bar{g}|_{\p M}$, the same coordinates are boundary normal coordinates w.r.t. $\bar{g}$. Thus locally the metrics are of the form 
$$g=g_{ij}dx'_idx'_j+dx_n^2,$$
$$\bar{g}=\bar{g}_{ij}dx'_idx'_j+dx_n^2,$$
where $i,j$ vary from $1$ to $n-1$. It suffices to prove that the normal derivatives are equal, i.e.
$$\p_n^m g_{ij}|_{x=x_0}=\p_n^m \bar{g}_{ij}|_{x=x_0},\, \, \p_n^m U|_{x=x_0}=\p_n^m \bar{U}|_{x=x_0}\, \,  \forall m=0,1,\cdots; i,j=1,\cdots, n-1.$$
We prove above equalities by induction, the case $m=0$ is granted. Assume for some nonnegative integer $l$ and all $0\leq m\leq l$, $\p_n^m g_{ij}|_{x=x_0}=\p_n^m \bar{g}_{ij}|_{x=x_0},\, \, \p_n^m U|_{x=x_0}=\p_n^m \bar{U}|_{x=x_0}$. Since $\p_n^{l+1}G_1|_{x=x_0}=\p_n^{l+1}\bar{G_1}|_{x=x_0}$, then
\begin{equation}\label{l+1}
(-\p_n^{l+1}U) g_{ij}+(k_1-U)\p_n^{l+1}g_{ij}=(-\p_n^{l+1}\bar{U}) \bar{g}_{ij}+(k_1-\bar{U})\p_n^{l+1}\bar{g}_{ij}
\end{equation}
at $x_0$. Similarly for energy $k_2$, since $\bar{g}=f_2^*g',\, \bar{U}=U'\circ f_2$ near $\p M$, we have at $x_0$ 
\begin{equation}
(-\p_n^{l+1}U) g_{ij}+(k_2-U)\p_n^{l+1}g_{ij}=(-\p_n^{l+1}\bar{U}) \bar{g}_{ij}+(k_2-\bar{U})\p_n^{l+1}\bar{g}_{ij}.
\end{equation}
Taking difference of above two equalities, we arrive
$$(k_1-k_2)\p_n^{l+1}g_{ij}|_{x=x_0}=(k_1-k_2)\p_n^{l+1}\bar{g}_{ij}|_{x=x_0}.$$
Since $k_1\neq k_2$, we obtain $\p_n^{l+1}g_{ij}|_{x=x_0}=\p_n^{l+1}\bar{g}_{ij}|_{x=x_0}$. Now return to the equation \eqref{l+1}, since $g|_{\p M}=\bar{g}|_{\p M}$ is positive definite,$\, U|_{\p M}=\bar{U}|_{\p M}$, we eventually get $\p_n^{l+1} U|_{x=x_0}=\p_n^{l+1} \bar{U}|_{x=x_0}$. This finishes the proof.
\end{proof}

\subsection{Scattering relation}
Now we show that for simple $\MP$-systems, the boundary rigidity problem is equivalent
to the problem of restoring a Riemannian metric, a magnetic potential and a potential from the scattering relations. Thus we will formulate all rigidity results in terms of the boundary rigidity problem in the next Section.

\begin{Theorem}\label{equivalence of two problems}
Suppose that $(g,\alpha,U)$ and $(g',\alpha',U')$ are simple
$\MP$-systems on $M$ of the same energy $k$ such that
$g|_{\p M}=g'|_{\p M}$, $U|_{\p M}=U'|_{\p M}$ and $\imath^* \alpha=\imath^* \alpha'$.
Then the boundary action functions $\mathbb A|_{\p M\times \p M}$
and $\mathbb A'|_{\p M\times \p M}$ of both the systems coincide if and only if
the scattering relations $\mathcal S$
and $\mathcal S'$  of these systems coincide.
\end{Theorem}
\begin{proof}
First, we introduce some notations. Let $G=2(k-U)g,\, G'=2(k-U')g'$, by Theorem \ref{MP system is magnetic system} and Proposition \ref{equivalence of simplicities}, $(G,\alpha)$ and $(G',\alpha')$ are simple magnetic systems of energy $\frac{1}{2}$. We denote by $\mathcal S_G$ and $\mathcal S_{G'}$ the scattering relations of $(G,\alpha)$ and $(G',\alpha')$ respectively (The definition of the scattering relation for a simple magnetic system is similar to that for a simple $\MP$-system by considering the magnetic flow of energy $\frac{1}{2}$). The notation $\p_+ SM$ denotes the bundle of inward unit vectors at $\p M$ with respect to metric $G$ (and also of $G'$, since $G|_{\p M}=G'|_{\p M}$).

Suppose $\mathbb A|_{\p M\times \p M}=\mathbb A'|_{\p M\times \p M}$, then by Proposition \ref{A=A_G} we have 
$$
\mathbb A_G|_{\p M\times \p M}=\mathbb A_{G'}|_{\p M\times \p M}.
$$
Then \cite[Lemma 2.5]{DPSU} implies that $\mathcal S_G=\mathcal S_{G'}$. Now we prove that this implies $\mathcal S=\mathcal S'$. Since the trajectories of $(g,\alpha,U)$ and $(G,\alpha)$ coincide, for any $(x,\xi)\in\p_+ S^kM$ the scattering relation $\mathcal S$ can be expressed in terms of $\mathcal S_G$ in the following way
$$
\mathcal S(x,\xi)=2\left[k-U\left(\mathfrak s_G\left(x,\frac{\xi}{2(k-U(x))}\right)\right)\right]\mathcal S_G\left(x,\frac{\xi}{2(k-U(x))}\right),
$$
where $\mathfrak s_G=\pi\circ\mathcal S_G$ ( Here we define $c(x,v)\doteq (x,cv)$ ). Exactly in the same way $\mathcal S'$ can be expressed in terms of $\mathcal S_{G'}$. Since $\mathcal S_G=\mathcal S_{G'}$, these expressions imply that $\mathcal S=\mathcal S'$.

Conversely, assume that $\mathcal S=\mathcal S'$. Since the trajectories of these two systems coincide, for any $(x,\xi)\in\p_+ SM$ the scattering relation $\mathcal S_G$ can be expressed in terms of $\mathcal S$ in the following way
$$
\mathcal S_G(x,\xi)=\frac{\mathcal S(x,2(k-U(x))\xi)}{2\left(k-U(\mathfrak s\left(x,2(k-U(x))\xi\right)\right))},
$$
where $\mathfrak s=\pi\circ\mathcal S$. Exactly in the same way $\mathcal S_{G'}$ can be expressed in terms of $\mathcal S'$. Since $\mathcal S=\mathcal S'$, these expressions imply that $\mathcal S_G=\mathcal S_{G'}$. Then \cite[Lemma 2.6]{DPSU} implies that
$$
\mathbb A_G|_{\p M\times \p M}=\mathbb A_{G'}|_{\p M\times \p M}.
$$
Applying Proposition \ref{A=A_G} we come to $\mathbb A|_{\p M\times \p M}=\mathbb A'|_{\p M\times \p M}$.
\end{proof}

\noindent {\it Remark:} Theorem \ref{equivalence of two problems} together with the counterexamples of the previous section shows that for generally a simple $\MP$-system, knowing the scattering relation of only one energy level is also insufficient for solving the scattering rigidity problem.


\section{Main results}
\subsection{Rigidity in a given conformal class}
Here we give the proof of our first main result which is a rigidity theorem in a fixed conformal class of a metric. The theorem below generalizes the corresponding
well-known results for the ordinary boundary rigidity problem, see \cite{Cr1, Mu, MR}, and for the magnetic boundary rigidity problem, see \cite{DPSU}.

\begin{Theorem}\label{conf}
Let $(g,\alpha,U)$ and $(g',\alpha',U')$ be simple
$\MP$-systems on $M$ with the same boundary action functions for both energy $k_1$ and $k_2$.
If $g'$ is conformal to $g$, then $g'=g$, $\alpha'=\alpha+d\varphi$ and $U'=U$ for some
smooth function $\varphi$ on $M$ vanishing on $\partial M$,
hence $(g',\alpha',U')$ is gauge equivalent to $(g,\alpha,U)$.
\end{Theorem}
\begin{proof}
Let $G_i=2(k_i-U)g,\, G'_i=2(k_i-U')g',\, i=1,2$, by Theorem \ref{MP system is magnetic system} and Proposition \ref{equivalence of simplicities}, $(G_i,\alpha)$ and $(G'_i,\alpha')$, for $i=1, 2$, are all simple magnetic systems of energy $\frac{1}{2}$. Let $\mathbb A_{G_i}$ and $\mathbb A_{G'_i}$ denote the Ma\~n\'e's action potentials (of energy $1/2$) for $(G_i,\alpha)$ and $(G'_i,\alpha')$ respectively. Then by Proposition \ref{A=A_G} we have $\mathbb A_{G_i}|_{\p M\times \p M}=\mathbb A_{G'_i}|_{\p M\times \p M}$. By the assumption $g'=\omega g$ for some strictly positive function $\omega\in C^\infty(M)$, therefore
$$
G'_i=\omega(k_i-U')(k_i-U)^{-1}G_i.
$$
Applying \cite[Theorem 6.1]{DPSU}, we get $G'_i=G_i$, i.e.
\begin{equation}\label{conf-eq1}
\omega(k_i-U')(k_i-U)^{-1}\equiv1,
\end{equation}
and that there are $\varphi_i\in C^{\infty}(M)$, with $\varphi_i|_{\p M}=0$, such that $\alpha'=\alpha+d\varphi_i$. But $\omega(k_i-U')(k_i-U)^{-1}\equiv1, \, i=1,2$ also impies that $$\frac{k_1-U'}{k_1-U}\equiv\frac{k_2-U'}{k_2-U}.$$
Thus $U=U'$ on $M$ (since $k_1\neq k_2$ ), together with \eqref{conf-eq1} this gives $\omega\equiv1$. On the other hand, $d\varphi_1=d\varphi_2$ with $\varphi_1|_{\p M}=\varphi_2|_{\p M}=0$ implies $\varphi_1=\varphi_2=\varphi$, thus $\alpha'=\alpha+d\varphi$ for some $\varphi\in C^{\infty}(M)$ with $\varphi|_{\p M}=0$.
\end{proof}

\noindent {\it Remark:} In Jollivet's paper on the scattering rigidity problem \cite{Jol}, the metrics $g$ and $g'$ are the same, namely the Euclidean metric, which means $\omega\equiv 1$ under the setting of Theorem \ref{conf}. Thus we have $(k-U')(k-U)^{-1}\equiv 1$, which implies $U=U'$ on $M$. That's why one fixed energy level is sufficient for Euclidean case. However, for general simple $\MP$-systems we need the information of two energy levels, as can be seen from the counterexamples and the proof above.

\subsection{Rigidity of real-analytic systems}
Our next result says that rigidity also holds in a class of real-analytic simple $\MP$-systems. This generalizes the corresponding result for the magnetic boundary rigidity problem in \cite{DPSU}.

\begin{Theorem}\label{analytic}
If $M$ is a real-analytic compact manifold with boundary,
and $(g,\alpha,U)$ and $(g',\alpha',U')$ are simple
real-analytic $\MP$-systems on $M$ with the same boundary action functions for both energy $k_1$ and $k_2$,
then these systems are gauge equivalent.
\end{Theorem}
\begin{proof}
Let $G=2(k_1-U)g$, $G'=2(k_1-U')g'$, by Theorem \ref{MP system is magnetic system} and Proposition \ref{equivalence of simplicities}, $(G,\alpha)$ and $(G',\alpha')$ are simple real-analytic magnetic systems of energy $\frac{1}{2}$. Let $\mathbb A_{G}$ and $\mathbb A_{G'}$ denote the Ma\~n\'e's action potentials (of energy $1/2$) for $(G,\alpha)$ and $(G',\alpha')$ respectively. Then by Proposition \ref{A=A_G} we have $\mathbb A_{G}|_{\p M\times \p M}=\mathbb A_{G'}|_{\p M\times \p M}$. Then \cite[Theorem 6.2]{DPSU} implies that $(G,\alpha)$ and $(G',\alpha')$ are gauge equivalent, i.e. there are some real-analytic diffeomorphism $f: M\rightarrow M$ with $f|_{\p M}=Id$ and some real-analytic function $\varphi$ on $M$ with $\varphi|_{\p M}=0$, such that $G'=f^*G=2(k_1-U\circ f)f^*g, \, \alpha'=f^*\alpha+d\varphi$. In particular, $g'$ is conformal to $f^*g$.

Now we consider the systems $(g',\alpha', U')$ and $(f^*g, \alpha', U\circ f)$. Let $\mathbb A_i$ and $\mathbb A'_i$ , $i=1, 2$, denote the Ma\~n\'e's action potentials (of energy $k_i$) for simple real-analytic $\MP$-systems $(g,\alpha, U)$ and $(g',\alpha', U')$ respectively. By our assumption, $\mathbb A'_i|_{\p M\times \p M}=\mathbb A_i|_{\p M\times \p M}=\bar{\mathbb A}_i|_{\p M\times \p M}, i=1,2$, where $\bar{\mathbb A}_i|_{\p M\times\p M}$ is the boundary action funciton of $(f^*g, \alpha', U\circ f)$ for energy $k_i$ (the second equality comes from the fact that $(g,\alpha, U)$ and $(f^*g,\alpha',U\circ f)$ are gauge equivalent) . Then, Theorem \ref{conf} impies that $U'=U\circ f$ and $g'=f^*g$.
\end{proof} 

\subsection{Rigidity of two-dimensional systems}
We show that two-dimensional simple $\MP$-systems are always rigid. Our result generalizes the boundary rigidity theorem
for simple Riemannian surfaces \cite{PU} and for simple two-dimensional magnetic systems \cite{DPSU}.

\begin{Theorem}\label{2-dim}
If\/ $\dim M=2$ and $(g, \alpha,U)$ and $(g',\alpha',U')$
are simple $\MP$-systems on $M$
with the same boundary action functions for both energy $k_1$ and $k_2$,
then these systems are gauge equivalent.
\end{Theorem}
\begin{proof}
Let $G=2(k_1-U)g$, $G'=2(k_1-U')g'$, by Theorem \ref{MP system is magnetic system} and Proposition \ref{equivalence of simplicities}, $(G,\alpha)$ and $(G',\alpha')$ are simple magnetic systems. Let $\mathbb A_{G}$ and $\mathbb A_{G'}$ denote the Ma\~n\'e's action potentials (of energy $1/2$) for $(G,\alpha)$ and $(G',\alpha')$ respectively. Then by Proposition \ref{A=A_G} we have $\mathbb A_{G}|_{\p M\times \p M}=\mathbb A_{G'}|_{\p M\times \p M}$. Applying \cite[Theorem 7.1]{DPSU} we find some diffeomorphism $f:M\to M$ with $f|_{\p M}=Id$, and a smooth function $\varphi:M\to \R$, with $\varphi|_{\p M}=0$, such that  $g'=(k_1-U\circ f)(k_1-U')^{-1}f^*g$ (i.e. $g'$ is conformal to $f^*g$) and $\alpha'=f^*\alpha+d\varphi$. Let $\mathbb A_i$ and $\mathbb A'_i$ denote the Ma\~n\'e's action potentials (of energy $k_i$) for simple $\MP$-systems $(g,\alpha, U)$ and $(g',\alpha', U')$ respectively. By our assumption, $\mathbb A'_i|_{\p M\times \p M}=\mathbb A_i|_{\p M\times \p M}=\bar{\mathbb A}_i|_{\p M\times \p M}, i=1,2$, where $\bar{\mathbb A}_i|_{\p M\times\p M}$ is the boundary action funciton of $(f^*g, \alpha', U\circ f)$ for energy $k_i$ (the second equality comes from the fact that $(g,\alpha, U)$ and $(f^*g,\alpha',U\circ f)$ are gauge equivalent) . Then, Theorem \ref{conf} impies that $U'=U\circ f$ and $g'=f^*g$. 
\end{proof}

\section{Final remark}

Our main results and the counterexamples have shown that it's necessary to consider two different energy levels for the boundary and scattering rigidity problems of simple $\MP$-systems. However, assuming the boundary action functions $\mathbb A=\mathbb A'$ for some fixed energy $k$, we still can obtain some weak version of boundary rigidity. 

After reviewing the proof of the main results, if two simple $\MP$-systems $(g, \alpha, U)$ and $(g', \alpha', U')$ have the same boundary action function for some energy $k$, then there exists a diffeomorphism $f: M\rightarrow M$ with $f|_{\p M}=Id$, and a smooth function $\varphi: M\rightarrow \R$ with $\varphi|_{\p M}=0$, such that $g'=(k-U')^{-1}(k-U\circ f)f^*g$ and $\alpha'=f^*\alpha+d\varphi$. Thus at least we can show that the magnetic potentials of these two $\MP$-systems are gauge equivalent, and the metrics of the two $\MP$-systems are gauge equivalent up to some conformal factor $(k-U')^{-1}(k-U\circ f)$, which is determined by the potentials of the two systems. In particular, $f=Id$ when $g$ is conformal to $g'$, and for the real-analytic $\MP$-systems, $f$ and $\varphi$ are both real-analytic. In some sense, this can be regarded as a weak boundary rigidity result, but the two systems may have different boundary action functions for energy levels other than $k$. However, if two simple $\MP$-systems are gauge equivalent, they must have the same boundary action functions for all $k>\sup_{M}U=\sup_MU'$.

Similar situation occurs for the scattering rigidity problem of simple $\MP$-systems.


\appendix
\section{}
\subsection{Ma\~n\'e's critical values}\label{Mane-action}
Here we adapt a certain part of the theory of convex superlinear Lagrangians
to the case of manifolds with boundary, see also \cite[Appendix A.1]{DPSU}.

Let $M$ be a compact Riemannian
manifold with boundary and let $L:TM\to\mathbb R$
be a $C^\infty$ Lagrangian satisfying the following hypotheses:

\begin{itemize}
\item {\em Convexity}: For all $x\in M$ the restriction of $L$ to $T_xM$
has everywhere positive definite Hessian.
\item {\em Superlinear growth}:
$$
\lim_{|v|\to\infty}\frac{L(x,v)}{|v|}=+\infty
$$
uniformly on $x\in M$.
\end{itemize}

The action of $L$ on an absolutely continuous curve
$\gamma:[a,b]\to M$ is
$$
\mathbb A_L(\gamma)=\int_a^b L(\gamma(t),\dot\gamma(t))\,dt.
$$

For each ${\lambda}\in \mathbb R$, the {\em Ma\~n\'e action potential}
$\mathbb A_{\lambda}:M\times M\to\mathbb R\cup \{-\infty\}$ is defined by
$$
\mathbb A_{\lambda}(x,y)=\inf_{\gamma\in \mathcal C(x,y)} \mathbb A_{L+{\lambda}}(\gamma),
$$
where
$\mathcal C(x,y)=\{\gamma:[0,T]\to M: \gamma(0)=x,\ \gamma(T)=y,\ \gamma
\text{ is absolutely continuous}\}$.

The {\em critical level} $c=c(L)$ is defined as
\begin{align*}
c(L)&=\sup\{\lambda\in\mathbb R: \mathbb A_{L+{\lambda}}(\gamma)<0\text{ for some closed curve }\gamma\}
\\
&=\inf\{\lambda\in\mathbb R: \mathbb A_{L+{\lambda}}(\gamma)\ge 0\text{ for every closed
curve }\gamma\}.
\end{align*}

Recall that the {\em energy function} $E:TM\to \mathbb R$ for $L$
is defined by
$$
E(x,v)=\de Lv(x,v)\cdot v-L(x,v),
$$
and that the energy function is constant on every solution $x(t)$ of the
Euler--Lagrange equation
\begin{equation}\label{e-l}
\frac d{dt}\de Lv(x(t),\dot x(t))=\de Lx(x(t),\dot x(t)).
\end{equation}

Let $\psi^t:TM\to TM$ be the Euler--Lagrange flow, defined
by $\psi^t(x,v)=(\gamma(t),\dot\gamma(t))$, where $\gamma$
is the solution of \eqref{e-l} with $\gamma(0)=x$ and $\dot\gamma(0)=v$.
For $x\in M$ and $k\in\mathbb R$, the {\em exponential map} at $x$
of energy $\lambda$
is defined to be the partial map
$\exp_x:T_xM\to M$ given by
$$
\exp^{\lambda}_x(tv)=\pi\circ\psi^t(v),\quad t\ge 0,\ v\in T_xM,\ E(x,v)=\lambda.
$$
Then $\exp^{\lambda}_x$  is a $C^1$-smooth partial map on $T_xM$
which is $C^\infty$-smooth on $T_xM\setminus\{0\}$.

The next two propositions were proved in \cite[Appendix A.1]{DPSU}.

\begin{Proposition}\label{kcl-simple}
If $\exp^{\lambda}_x:(\exp^{\lambda}_x)^{-1}(M)\to M$
is a diffeomorphism for every $x\in M$, then ${\lambda}\ge c(L)$.
\end{Proposition}

\begin{Proposition}\label{minimize-L}
If ${\lambda}>c(L)$ and $x,y\in M$ $x\ne y$, then
there is $\gamma\in \mathcal C(x,y)$ such that
$$
\mathbb A_{\lambda}(x,y)=\mathbb A_{L+{\lambda}}(\gamma).
$$
Moreover, the energy of $\gamma$ is $E(\gamma,\dot\gamma)\equiv{\lambda}$.
\end{Proposition}

Now, we apply the above to the case of $\MP$-systems.
For a simple $\MP$-system $(M,g,\alpha,U)$,
the $\MP$-flow can also be obtained as the Euler--Lagrange flow with
the corresponding Lagrangian defined by
$$
L(x,v)=\frac 12|v|^2_g-\alpha_x(v)-U(x).
$$

\begin{Lemma}\label{minimize}
Let $(g,\alpha,U)$ be a simple $\MP$-system on $M$.
For $x,y\in M$, $x\ne y$,
\begin{equation*}
\mathbb A_k(x,y)=\mathbb A_{L+k}(\gamma_{x,y})
=2kT_{x,y}-\int_{\gamma_{x,y}}(\alpha+2U),
\end{equation*}
where $\gamma_{x,y}:[0,T_{x,y}]\to M$ is the $\MP$-geodesic with constant energy $k$ from $x$ to $y$.
\end{Lemma}

\begin{proof}
It is easy to see that the simplicity assumption implies that
for this Lagrangian the assumptions of Proposition \ref{kcl-simple}
hold for all $\lambda$ sufficiently close to $k$. Therefore,
the proposition gives $k>c(L)$. Then Proposition \ref{minimize-L}
shows that, given $x\ne y$ in $M$, there is $\gamma\in\mathcal C(x,y)$
with energy $k$ such that $\mathbb A(x,y)=\mathbb A(\gamma)$.
Using simplicity, one can then prove that $\gamma$ is a $\MP$-geodesic with constant energy $k$, i.e., $\gamma=\gamma_{x,y}$.
\end{proof}

\subsection{$\MP$-convexity}
Let $M$ be a compact manifold with boundary,
endowed with a Riemannian metric~$g$, a closed $2$-form $\Omega$ and a smooth function $U$. Consider a manifold $M_1$ such that $M_1^\text{int}\supset M$.
Extend $g$, $\Omega$ and $U$ to $M_1$ smoothly, preserving the former notation
for extensions.
We say that $M$ is
{\em $\MP$-convex} at $x\in \partial M$
if there is a neighborhood $O$ of $x$ in $M_1$ such that
all $\MP$-geodesics of constant energy $k$ in $O$, passing through $x$
and tangent to $\partial M$ at $x$,
lie in $M_1\setminus M^{\text{int}}$.
If, in addition, these geodesics
do not intersect $M$ except for $x$,
we say that $M$ is {\em strictly $\MP$-convex} at $x$.
It is not hard to show that these definitions depend neither on the
choice of $M_1$ nor on the way we extend $g$, $\Omega$ and $U$ to $M_1$.

As before, we let $\Lambda$ denote the second fundamental form of $\partial M$
and $\nu(x)$ the inward unit vector normal to $\partial M$ at $x$.

\begin{Lemma}\label{convexity}
If $M$ is $\MP$-convex at $x\in\partial M$, then
\begin{equation}\label{nonstrict-conv}
\Lambda(x,v)\ge\langle Y_x(v),\nu(x)\rangle-d_xU(\nu(x))
\quad\text{for all } v\in S^k_x(\partial M).
\end{equation}
If the inequality is strict, then $M$ is strictly $\MP$-convex at $x$.
\end{Lemma}

\begin{proof}
Suppose $M$ is $\MP$-convex at $x$.
Choosing a smaller $O$ if necessary, we may assume
that there is a smooth function $\rho$ on $O$ such that $|\grad\rho|=1$
and $\partial M\cap O=\rho^{-1}(0)$. Further we may assume that all the above
$\MP$-geodesics lie in $O^-=\{x:\rho(x)\le 0\}$.

Let $v\in S^k_x (\partial M)$ and $\gamma(t)$ be the
$\MP$-geodesic with $\gamma(0)=x$, $\dot \gamma(0)=v$. By assumption,
$\rho\circ \gamma(t)\le 0$ for all small $t$. Therefore,
$$
\frac{d^2}{dt^2}\big[\rho\circ\gamma(t)\big]\Big|_{t=0}\le 0.
$$
Since
\begin{multline*}
\frac{d^2}{dt^2}\big[\rho\circ\gamma(t)\big]
=\frac{d}{dt}\langle \grad\rho(\gamma(t)),\dot\gamma(t)\rangle
\\
=\langle \nabla_{\dot\gamma(t)}\grad\rho(\gamma(t)),\dot\gamma(t)\rangle
+\langle \grad\rho(\gamma(t)),\ddot\gamma(t)\rangle
\\
=\hess_{\gamma(t)}\rho (\dot\gamma(t),\dot\gamma(t))+\langle \grad\rho(\gamma(t)),Y(\dot\gamma(t))-\nabla U(\gamma(t))\rangle
\end{multline*}
and since $\Lambda(x,v)=-\hess_x\rho(v,v)$ and $\grad\rho(x)=\nu(x)$
when $(x,v)\in S^k(\p M)$, we arrive at \eqref{nonstrict-conv}.

Now, assume that \eqref{nonstrict-conv} is strict.
Then there is $\delta>0$ such that for every $\MP$-geodesic $\gamma$
in $M_1$ with $\gamma(0)=x$ and $\dot\gamma(0)=v\in S^k_x(\partial M$),
$$
\frac{d^2}{dt^2}\big[\rho\circ\gamma(t)\big]\Big|_{t=0}\le -\delta.
$$
Thus, there is a small $\varepsilon>0$ such that
$$
\rho\circ\gamma(t)\le -\frac 14\delta t^2
\quad \text{for all }t\in (-\varepsilon,\varepsilon).
$$
This implies the second statement.
\end{proof}

\subsection{Scattering relation}
For $(x,\xi)\in \p_+ S^k M$, let $\tau(x,\xi)$ be the time when the $\MP$-geodesic $\gamma_{x,\xi}$, such that $\gamma_{x,\xi}(0)=x$,
$\dot\gamma_{x,\xi}(0)=\xi$, exits. Clearly, the function $\tau(x,\xi)$ is continuous and,
using the implicit function theorem, it is easily seen to be smooth near a point $(x,\xi)$ such that
the $\MP$-geodesic $\gamma_{x,\xi}(t)$ meets $\p M$
transversally at $t=\tau(x,\xi)$. By  \eqref{strict-conv} and Lemma \ref{convexity} in Appendix A,
the last condition holds everywhere on $\p_+S^kM\setminus S^k(\p M)$.
Thus, $\tau$ is a smooth function on $\p_+S^kM\setminus S^k(\p M)$.

\begin{Lemma}\label{l-l}
For a simple $\MP$-system, the function $\tau:\partial_+S^kM\to\mathbb R$ is smooth.
\end{Lemma}

\begin{proof}
Let $\rho$ be a smooth nonnegative function on $M$
such that $\partial M=\rho^{-1}(0)$
and $|\grad\rho|=1$ in some neighborhood of $\partial M$.
Put $h(x,\xi,t)=\rho(\gamma_{x,\xi}(t))$ for $(x,\xi)\in\p_+S^kM$.
Then
\begin{align*}
&h(x,\xi,0)=0,
\\
&\de ht(x,\xi,0)=\langle \nu(x),\xi\rangle,
\\
&
\frac{\partial^2 h}{\partial t^2}(x,\xi,0)
=\hess_x\rho (\xi,\xi)+\langle \nu(x),Y(\xi)-\nabla U(x)\rangle.
\end{align*}

Therefore, for some smooth function $R(x,\xi,t)$,
$$
h(x,\xi,t)=\langle \nu(x),\xi\rangle t
+\frac12\left(\hess_x\rho (\xi,\xi)+\langle \nu(x),Y(\xi)-\nabla U(x)\rangle\right)t^2
+R(x,\xi,t)t^3.
$$
Since $h(x,\xi,\tau(x,\xi))=0$, it follows that $L=\tau(x,\xi)$
is a solution of the equation
\begin{equation}\label{equation-l}
F(x,\xi,L):=\langle \nu(x),\xi\rangle
+\frac12\left(\hess_x\rho (\xi,\xi)+\langle \nu(x),Y(\xi)-\nabla U(x)\rangle\right)L
+R(x,\xi,t)L^2=0.
\end{equation}
By \eqref{strict-conv}, for $(x,\xi)\in S^k(\partial M)$
\begin{align*}
\de FL(x,\xi,0)
&=\frac12\left(\hess_x\rho (\xi,\xi)+\langle \nu(x),Y(\xi)-\nabla U(x)\rangle\right)
\\
&=\frac12\left(-\Lambda(x,\xi)+\langle \nu(x),Y(\xi)-\nabla U(x)\rangle\right)<0.
\end{align*}
Now, the implicit function theorem yields smoothness of $\tau(x,\xi)$
in a neighborhood of $S^k(\p M)$. Since $\tau$ is also smooth
on $\p_+S^kM\setminus S^k(\p M)$, we conclude that
$\tau$ is smooth on $\p_+S^kM$.
\end{proof}

\end{document}